\documentclass{article}

\usepackage{arxiv}

\usepackage[utf8]{inputenc} % allow utf-8 input
\usepackage[T1]{fontenc}    % use 8-bit T1 fonts
\usepackage{hyperref}       % hyperlinks
\usepackage{url}            % simple URL typesetting
\usepackage{booktabs}       % professional-quality tables
\usepackage{amsfonts}       % blackboard math symbols
\usepackage{nicefrac}       % compact symbols for 1/2, etc.
\usepackage{microtype}      % microtypography
\usepackage{lipsum}
\usepackage{amsmath}

\usepackage{amsfonts,amssymb}
\usepackage{graphicx}
\usepackage{amsthm}
\usepackage{xcolor}

\newcommand{\R}{\mathbb{R}}

\newtheorem{theorem}{Theorem}[section]
\newtheorem{proposition}{Proposition}[section]
\newtheorem{lemma}{Lemma}[section]
\newtheorem{corollary}{Corollary}[section]

\newtheorem{remark}{Remark}[section]

\usepackage{cite}
\usepackage{authblk}
\def\dom{\text{dom}}
\def\sign{\text{sgn}\,}
\newcommand{\p}{\partial}
\newcommand{\bb}{\begin{equation}}
\newcommand{\ee}{\end{equation}}
\newcommand{\ba}{\begin{array}}
\newcommand{\ea}{\end{array}}
\newcommand{\f}{\frac}
\usepackage[all]{xy}
\newcommand{\ds}{\displaystyle}

\newcommand{\al}{\alpha}
\newcommand{\be}{\beta}

\numberwithin{equation}{section}

\title{Wave breaking for shallow water models with time decaying solutions}

\author[1,2]{Igor Leite Freire}
\affil[1]{Mathematical Institute, Silesian University in Opava,
Na Rybn\'\i{}\v{c}ku, 1, 74601, Opava, Czech Republic}
\affil[2]{Centro de Matem\'atica, Computa\c{c}\~ao e Cogni\c c\~ao, Universidade Federal do ABC, Avenida dos Estados, $5001$, Bairro Bangu,
$09.210-580$, Santo Andr\'e, SP - Brasil\\
  \texttt{igor.freire@ufabc.edu.br} \\
  \texttt{igor.leite.freire@gmail.com}}

\begin{document}
\maketitle
\begin{abstract}
\centering\begin{minipage}{\dimexpr\paperwidth-9cm}
\textbf{Abstract:} A family of Camassa-Holm type equations with a linear term and cubic and quartic nonlinearities  is considered. Local well-posedness results are established via Kato's approach. Conserved quantities for the equation are determined and from them we prove that the energy functional of the solutions is a time-dependent, monotonically decreasing function of time, and bounded from above by the Sobolev norm of the initial data under some conditions. The existence of wave breaking phenomenon is investigated and necessary conditions for its existence are obtained. In our framework the wave breaking is guaranteed, among other conditions, when the coefficient of the linear term is sufficiently small, which allows us to interpret the equation as a linear perturbation of some recent Camassa-Holm type equations considered in the literature.

\vspace{0.1cm}
\textbf{2010 AMS Mathematics Classification numbers}: 35A01, 35L65, 37K05.

\textbf{Keywords:} Camassa-Holm type equations, Kato's approach, wave breaking, time dependent norms.

%\textbf{Dedicatory:} This paper is dedicated to Professor Nail Ibragimov.

\end{minipage}
\end{abstract}
\bigskip
\newpage
\tableofcontents
\newpage

\section{Introduction}

Indubitably the Camassa-Holm (CH) equation is one of the cutting edge topics in mathematics and mathematical physics in general, and differential equations and applied analysis in particular.

The work of Camassa and Holm \cite{chprl}, where the equation was derived as a model for shallow water waves, can be considered as of fundamental nature and is the cornerstone of an active field of investigation of {\it non-local integrable equations}. The discovery of the CH equation was the prelude of a myriad of other non-local integrable equations, such as the Degasperis-Procesi (DP) \cite{depro,DHH} and the Novikov equations \cite{novikov,hone}. They share the following properties with the CH equation: bi-Hamiltonian structure, integrability, existence of peakon and multipeakon solutions, and infinite hierarchy of conservation laws, to name a few.

The CH equation was firstly deduced by Fuchsteiner and Fokas in \cite{fokas}. It was, in fact, a formal deduction, without physical motivation. However, on the grounds of physics, in \cite{chprl} the equation was recovered as a unidirectional model for describing the height of water's free surface above the flat bottom for a shallow water system, see also \cite{john}. 

One of the interesting properties of the CH equation, and other similar equations not necessarily among the ones mentioned before, is the fact that they have many conservation laws. This is very useful and vital to establish qualitative results about the solutions of equations and partially explains why this sort of equations are particularly fashion in the field of differential equations.

We would like to recall some simple, but extremely important concepts in differential equations where time is involved. Given a differential equation and an initial data (that is, a Cauchy problem), some questions are of capital importance:
\begin{enumerate}
    \item Does the problem have a solution? This is the question of existence.
    \item If there is a solution, is it unique? This is the problem of uniqueness.
    \item If there is a solution, does it depend continuously on the initial data?
\end{enumerate}

These three essential questions are central when one is investigating qualitative behaviour of the solutions of equations subject to the restrictions imposed by a certain condition. In this paper, the equation we shall study depends on $t$ and $x$ and the condition we shall consider is an initial condition, that is, we know how the solution behaves at $t=0$. In case these questions are satisfied, the problem under consideration is said to be well-posed (in the sense of Hadamard).

Although the questions above are basic and fundamental, they are not necessarily simple or easy to be addressed. In particular, very often the solution $u$ of a problem like the one we shall deal with in this paper belongs to $C^0([0,T),E)$, meaning that $u(\cdot,x)\in C^0([0,T))$ and $u(t,\cdot)\in E$, where $E$ is a usually a Banach space, whereas $T>0$ is the {\it maximal time of existence of the solution}, or {\it lifespan}, which usually depends on the initial data and the space $E$.

The value of the lifespan add another ingredient to our menu: If the problem is well-posed, what happens in case $T<\infty$? What if $T=\infty$? In the first case we have a local well-posed problem: it is well-posed provided that $t\in[0,T)$, while in the second we have a global well-posed problem, meaning that the solution exists for any $t$.

In case we only have local well-posedness, we say that the solution $u$ of the problem develops a finite time blow-up. This happens just because the solution cannot be described for all values of $t$. 

The blow-up phenomena can have different manifestations, depending on the problem. It can arise, for example, if the solution of the problem becomes unbounded when $t$ approximates $T$. More specifically, a blow-up in finite time turns out if $T<\infty$ and
$$\lim\sup\limits_{t\rightarrow T}\|u(t,\cdot)\|_E=\infty.$$

A blow-up can also arise in the following situation: Assume that the Banach space $E$ is a subspace of $C^1(\R)$ and $T<\infty$. Then we have another sort of blow-up if
$$\sup_{(t,x)\in[0,T)\times\R}|u(t,x)|<\infty,\quad\text{but}\quad\lim\sup\limits_{t\rightarrow T}\left(\sup_{x\in\R}|u_x(t,x)|\right)=\infty.$$

This kind of blow-up, which takes the $x-$derivative of the solution into account, is best known as {\it wave breaking}. From a geometrical viewpoint, it means that the tangent line to the curve $x\mapsto (x,u(t,x))$ tends to the perpendicular line to the $x-$direction when $t$ approaches $T$.

\subsection{Notation and conventions}
Throughout this paper, if $u$ is a function depending on two variables, both $u_t$ and $\p_t u$ denote the derivative of $u$ with respect to its first argument, while $u_x$ and $\p_x u$ mean the derivative with respect to the second independent variable.

The norm in a Banach space $E$ is denoted by $\|\cdot\|_E$, whereas $\langle \cdot,\cdot\rangle _{H}$ means the inner product in a Hilbert space $H$. If $E$ and $F$ are two Banach spaces, the set of bounded linear operators from $E$ into $F$ is denoted by ${\cal L}(E,F)$, or ${\cal L}(E)$ in the case $F=E$.

The space of test functions $\varphi:\R\rightarrow\R$ is denoted by ${\cal S}(\R)$, whereas its dual topological space is refereed as ${\cal S}'(\R)$. A member of ${\cal S}(\R)$ is known as rapidly decreasing smooth function, while members of ${\cal S}'(\R)$ are called tempered distributions.

\subsection{Main results of the paper and its outline}

Let us introduce the subject of investigation of the present paper. Our main interest here is the equation
\bb\label{1.0.1}
u_t-u_{txx}+3uu_x+\lambda(u-u_{xx})=2u_xu_{xx}+uu_{xxx}+\al u_x+\beta u^2u_x+\gamma u^3u_x+\Gamma u_{xxx}.
\ee

It reduces to the CH equation if $\lambda=\al=\be=\gamma=\Gamma=0$, whereas the Dullin-Gottwald-Holm equation \cite{dgh,raspa} is recovered when $\lambda=\be=\gamma=0$ and $\al\Gamma\neq0$. If $\al=\beta=\gamma=\Gamma=0$ and $\lambda>0$ we have the weakly dissipative CH equation \cite{wujmp2006,wujde2009}, whereas if $\lambda>0$, $\al\Gamma\neq0$ and $\be=\gamma=0$ we have the weakly dissipative DGH equation \cite{novjmp2013,novjde2016,zhangjmp2015}.

If $\lambda=0$, equation \eqref{1.0.1} includes a shallow water model with Coriolis effects proposed in \cite{chines-adv,chines-arxiv,gui-jnl, chines-jde}, see also \cite{silvajde2019,silva2019}.

Making use of the auxiliary variable $m=u-u_{xx}$ we reformulate the object of investigation of this paper as the following: We aim at investigating properties of the initial value problem,
\bb\label{1.0.2}
\left\{
\ba{l}
m_t+um_x+2u_xm+\lambda m=\al u_x+\beta u^2u_x+\gamma u^3u_x+\Gamma u_{xxx},\quad t>0,\,\,\,x\in\R,\,\,\,\lambda\neq0,\\
\\
u(0,x)=u_0(x).
\ea
\right.
\ee

Our first aspiration is to address the three questions listed in the beginning of the paper. We prove that \eqref{1.0.1} is locally well-posed if the initial data belongs to the Sobolev space $H^s(\R)$, with $s>3/2$. Namely, we have the following result:
\begin{theorem}\label{teo1.1}
Given $u_0\in H^{s}(\R)$, $s>3/2$, then there exist a maximal time $T=T(u_0)>0$ and a unique solution $u$ to \eqref{1.0.1} satisfying the initial condition $u(0,x)=u_0(x)$, such that $u=u(\cdot,u_0)\in C^{0}([0,T);H^{s}(\R))\cap C^{1}([0,T);H^{s-1}(\R))$. Moreover, the solution depends continuously on the initial data, in the sense that the map $u_0\mapsto u(\cdot,u_0):H^{s}(\R)\rightarrow C^{0}([0,T);H^{s}(\R))\cap C^{1}([0,T);H^{s-1}(\R))$ is continuous and $T$ does not depend on $s$.
\end{theorem}
Theorem \ref{teo1.1} is proved using Kato's approach \cite{kato} and its demonstration is presented in Section \ref{sec2}.

In order to enlighten further qualitative analysis of the solutions of \eqref{1.0.1}, the next proposition gives us some differential identities.

\begin{proposition}\label{prop1.1}
Let $v=v(t,x)$ be a function such that $v\in C^1(\R^2,\R)$, its second and third order derivatives exist, $v_{tx}=v_{xt}$, and
$$E:=v_t-v_{txx}+\lambda(v-v_{xx})-2v_xv_{xx}-vv_{xxx}-(\alpha+\beta v^2+\gamma v^3-3v)v_{x}+\Gamma v_{xxx},$$
where $\al,\,\be,\,\gamma$ and $\Gamma$ are constants. Then the following formal identities holds:
\bb\label{3.0.1}
\ba{lcl}
E&=&\ds{\lambda v+\p_t v+\p_x\left(\f{3}{2}v^2-v_{tx}-vv_{xx}-v_x^2-\al v -\f{\be}{3}v^3-\f{\gamma}{4}v^4-\Gamma v_{xx}-\lambda v_{x}\right)},
\ea
\ee
and
\bb\label{3.0.2}
\ba{lcl}
vE&=&\ds{\lambda(v^2+v_x^2)+\p_ t\left(\f{v^2+v_x^2}{2}\right)}\\
\\
&&\ds{+\p_x\left(v^3-v^2v_{xx}-vv_{tx}+\f{\Gamma}{2}v_x^2-\Gamma vv_{xx}-\f{\al}{2}v^2-\f{\beta}{4}v^4-\f{\gamma}{5}v^5-\lambda vv_x\right)}.
\ea
\ee
In the very particular case $\be=\gamma=0$ and $\Gamma=-\al$, we have a third identity:
$$%\bb\label{3.0.3}
\f{1}{2}(v-v_{xx})^{-1/2}E=\f{\lambda}{2} (v-v_{xx})^{1/2}+\p_t(v-v_{xx})^{1/2}+\p_x[(-\al)(v-v_{xx})^{1/2}].
$$%\ee
\end{proposition}
The proof of Proposition \ref{prop1.1} is straightforward and, for this reason, is omitted. However, these identities, when considered on the solutions of the equation \eqref{1.0.2}, give us very useful invariants. 
\begin{theorem}\label{teo1.2}
Assume that $u\in C^1(\R^2;\R)$ is a solution of \eqref{1.0.1} such that $u,\,u_x\rightarrow0$ as $x\rightarrow\pm\infty$, $u_{tx}=u_{xt}$ and its second order derivatives are bounded. Let
$$%\bb\label{3.0.4}
{\cal H}_0(t):=\int_\R u(t,x)\,dx
$$%\ee
and
\bb\label{3.0.5}
{\cal H}_1(t):=\f{1}{2}\int_\R (u(t,x)^2+u_x(t,x)^2)\,dx.
\ee
Then, for any $t$, we have ${\cal H}_0(t)=e^{-\lambda t}{\cal H}_0(0)\quad\text{and}\quad {\cal H}_1(t)=e^{-2\lambda t}{\cal H}_1(0)$. In particular, $\|u\|_{H^1(\R)}^2=2{\cal H}_1(t)$, and if $\lambda\geq0$, then
$\|u\|_{H^1(\R)}= e^{-\lambda t}\|u_0\|_{H^1(\R)}\leq \|u_0\|_{H^1(\R)}$.
\end{theorem}
We shall refer to \eqref{3.0.5} generically as {\it energy} (functional) and it is nothing but a time dependent $H^1(\R)-$norm of the solution $u$ (which we shall simply refer as Sobolev norm for convenience). If $\lambda<0$ we easily see that its norm increases as the increasing of $t$, leading to unbounded solutions on $[0,\infty)\times\R$, whereas the case $\lambda>0$ is more interesting, since in this situation the Sobolev norm for any $t$ is bounded from above by the Sobolev norm of the initial data.

Another qualitative property of the solutions is given in the next result.

\begin{theorem}\label{teo1.3}
Let $u$ be a solution to \eqref{1.0.1} such that $u_{tx}=u_{xt}$, and $u,\,u_x,\,u_{xx}$ are integrable and vanishing as $x\rightarrow\pm\infty$ for any value of $\lambda$. If $u_0(x):=u(0,x)$, $m=u-u_{xx}$ and $m_0:=u_0-u_0''$, then
\bb\label{3.0.6}
e^{-\lambda t}\int_\R m_0 dx=\int_\R m dx=\int_{\R}u dx=e^{-\lambda t}\int_\R u_0 dx.
\ee
\end{theorem}

Theorems \ref{teo1.2} and \ref{teo1.3} are proved in Section \ref{sec3}.

A natural question is whether \eqref{1.0.1} admits wave breaking, that can be assured in the following scenario, which will be proved in Section \ref{sec4}:
\begin{theorem}\label{teo1.4}
Let $u_0\in H^3(\R)$, $u=u(t,x)$ be the corresponding solution to \eqref{1.0.2}, $y(t):=\inf\limits_{x\in\R}u_x(t,x)$, and $\kappa:=\max\{|\alpha|,|\beta|/3,|\gamma|/4,|\Gamma|\}$. Assume that there exists $\theta\in(0,\theta_0]$ and $x_0\in\R$ such that $\theta u_0'(x_0)<\min\{-\|u_0\|^{1/2}_{H^1(\R)},-\|u_0\|^2_{H^1(\R)}\}$, where 
$\theta_0:=\sqrt{2/(1+12\kappa)}.$

If $$\lambda\in\left(0,-\f{y(0)}{4}\f{\theta^2 u_0'(x_0)^2-\max\{\|u_0\|_{H^1(\R)},\|u_0\|_{H^1(\R)}^4\}}{\theta^2 u_0'(x_0)^2}\right),$$
then wave breaking for \eqref{1.0.2} occurs.
\end{theorem}

Theorem \ref{teo1.4} can only foresee the emergence of wave breaking of the solutions of \eqref{1.0.2} for limited, but positive, values of $\lambda$.  We observe that if  $\lambda$ in \eqref{1.0.2} is small, then we can interpret the term $\lambda\,m$ as a perturbation of the equation
\bb\label{1.0.3}
\left\{
\ba{l}
m_t+um_x+2u_xm=\al u_x+\beta u^2u_x+\gamma u^3u_x+\Gamma u_{xxx},\quad t>0,\,\,\,x\in\R,\\
\\
u(0,x)=u_0(x),
\ea
\right.
\ee
which was studied in \cite{silvajde2019,silva2019}, and also in \cite{chines-adv,chines-arxiv,gui-jnl,chines-jde} for specific choices of the parameters $\al,\,\be,\,\gamma$ and $\Gamma$, as already mentioned.

Although we may unriddle the presence of the term $\lambda m$ in \eqref{1.0.1} as a perturbation of \eqref{1.0.3}, we cannot underrate it, since its presence brings structural and substantial changes in the behaviour of the solutions of equation \eqref{1.0.1}. As we shall show in section Section \ref{sec3}, it implies that if global solutions of \eqref{1.0.2} exist, as well as their corresponding energy functional, then they vanish as $t\rightarrow\infty$. Moreover, for $\lambda$ sufficiently small, the conditions for wave breaking of the solutions \eqref{1.0.2} are unaltered when compared to \eqref{1.0.3}, as one can observe comparing our results with those proved in \cite{silva2019} regarding this matter.

The proof of local well-posedness naturally leads to the question of whether the problem \eqref{1.0.2} has global solutions. While we have an accurate description of the manifestation of blow-up phenomena as wave breaking, the question of global solvability of equation \eqref{1.0.1} is unclear. In Section \ref{sec5} we bring some light to this problem and we show the limitations and open problems that lead us to fail in giving a complete description for the global existence of solutions of \eqref{1.0.1}.

We discuss our results in Section \ref{sec6}, while in Section \ref{sec7} we present our conclusions.

\subsection{ Novelty and challenges of this paper}

The presence of the cubic and quartic nonlinearities in \eqref{1.0.2} brings some difficulties in the qualitative analysis of the solutions of \eqref{1.0.2} when compared with similar works dealing with \eqref{1.0.3}. This is somewhat expected, since the presence of these higher order nonlinearities introduces substantial modifications on the behaviour of the solutions of \eqref{1.0.2} in comparison with \eqref{1.0.3}, as one can infer by comparing\footnote{It is worth mentioning that in all of these works there is no linear term in the equations involved.} the results in \cite{const1998-1,const1998-2,const1998-3,const2000-1,raspa,escher,lenjde2013,liu2011,blanco,tiancmp2005,yindcdc2004} with those in \cite{silvajde2019,silva2019,chines-adv,chines-arxiv,gui-jnl, chines-jde}. In particular, the equation \eqref{1.0.3} does not seem to be integrable unless $\be=\gamma=0$, see \cite{silvajde2019,silva2019}. Turning back to our case, if $\lambda\beta\neq0$ or $\lambda\gamma\neq0$ we have a dramatically different situation to be considered regarding both global existence and wave breaking phenomena. Besides, if one of these conditions is satisfied, then we cannot reduce the analysis of equation \eqref{1.0.2}, with $\lambda\neq0$, to equation \eqref{1.0.3} as it is possible for some dissipative CH type equations, as pointed out in \cite{lenjde2013}.

In case $\al=-\Gamma$, $\lambda>0$ and $\be=\gamma=0$, the equation \eqref{1.0.3} is reduced to
$$m_t+um_x+2u_xm+\lambda m=\al m_x,$$
which, up to notation, was investigated in \cite{novjmp2013,zhangjmp2015}. The change $(u(t,x)\mapsto u(t,x+\al t)$ transform the last equation into the weakly dissipative CH equation, which was investigated in \cite{wujmp2006,wujde2009}.

It was pointed out by Lenells and Wunsch \cite{lenjde2013} that the analysis of the latter equation can be reduced to the CH equation through the change of variables
$$
u(t,x)\mapsto e^{-\lambda t}u\left(\f{1-e^{-\lambda t}}{\lambda},x\right).
$$

In our case, if $\al\neq-\Gamma$ or $(\be,\gamma)\neq(0,0)$, then we cannot use the results in \cite{lenjde2013} to reduce the analysis of equation \eqref{1.0.1} to other previously known. %Let us explain the reason for the scalar case (in \cite{lenjde2013} some systems are also considered and our observations can directly be extended to them).

The clever observation made by Lenells and Wunsch \cite{lenjde2013} can be applied to a large class of relevant equations in the field of shallow water models and integrable systems, in particular, to some of the references cited above (and perhaps in others in the list) regarding (weakly) dissipative equations. As one can infer from the results in \cite{lenjde2013}, the transformation connecting equations are of the type
\bb\label{6.0.1}
u(t,x)\mapsto e^{-\lambda t}u\left(\f{1-e^{-(p-1)\lambda t}}{\lambda},x\right),
\ee
where $p$ is a certain parameter. The reason for this transformation works is the following (we pay attention only to the scalar cases, but the explanation for systems follows the same argument): The equations considered in \cite{lenjde2013} are of the form $m_t+\lambda m=F[u,m]$, where $F$ is a homogeneous polynomial in $u,\,m$ and their derivatives with respect to $x$. This means that if $\mu$ is a parameter, then $F[\mu u,\mu m]=\mu^p F[u,m]$, for some $p$. It is well known that the CH, DP, Novikov and other similar equations are scale-invariant, see \cite{anco,boz,pri-book,pri-aims} and references therein. In particular, $p=2$ for the CH equation and its corresponding dissipative equation. The homogeneity of $F[u,m]$ allows us to use the substitution \eqref{6.0.1} to reduce the equations mentioned in \cite{lenjde2013} to their counterparts with $\lambda=0$.

In our case, if $\al\neq-\Gamma$ we cannot eliminate the term $\al u_x+\Gamma u_{xxx}$. On the other hand, if $(\be,\gamma)\neq(0,0)$ then we have the presence of cubic and quartic terms. In any situation, we lose homogeneity and, therefore, the elegant transformation introduced in \cite{lenjde2013} is no longer admissible. Moreover, these higher order terms not only prevent us to invoke the results aforementioned, but also bring difficulties to completely describe conditions for the global existence of solutions for the Cauchy problem \eqref{1.0.2}.

Another peculiarity of the present paper is the fact that we construct some invariants for the solutions of the equation with enough decaying at infinity through the establishment of differential identities, as shown in Proposition \ref{prop1.1}. These invariants will provide us estimates on the solutions of the equation that will help us to describe the wave breaking scenario.

\section{Local well-posedness}\label{sec2}

Here we establish the existence and uniqueness of solutions of the Cauchy problem \eqref{1.0.2}. More precisely, we respond questions 1, 2 and 3 in the Introduction of the paper. These results are established at the local level, meaning that we guarantee the existence of $T>0$ such that the solution exists on $[0,T)\times\R$. 

\subsection{Overview of functional analysis}\label{subsec2.1}

We present a short overview of Sobolev spaces, embeddings and mappings between Sobolev spaces. For further readings about these subjects, the reader is referred to \cite{brezis,folland,hunter,linares,taylor}.

If $\phi$ is a tempered distribution, its Fourier transform ${\cal F}(\phi)$ and its corresponding inverse are, respectively, given by
$$
\hat{\phi}(\xi):={\cal F}(\phi)(\xi)=\f{1}{\sqrt{2\pi}}\int_{-\infty}^{+\infty}\phi(x)e^{-ix\xi}dx\quad\text{
and}\quad 
\phi(x)={\cal F}^{-1}(\hat{\phi})(x):=\f{1}{\sqrt{2\pi}}\int_{-\infty}^{+\infty}\hat{\phi}(\xi)e^{ix\xi}d\xi.
$$

Given $s\in\R$, the Sobolev space of order $s$ is given by
$H^{s}(\R)=\Lambda^{-s}(L^2(\R)),$
where $L^2(\R)(=H^0(\R))$ denotes the usual Hilbert space of the squared integrable functions and
$\Lambda^{s}u:={\cal F}((1+|\xi|^2)^{s/2}\hat{u}).$

The operator $\Lambda^{s}$ is a unitary isomorphism between $H^{t}(\R)$ and $H^{t-s}(\R)$, for any $t\in\R$. Of particular importance is the case $s=2$, in which the operator can be identified with the differential operator $\Lambda^2:=1-\p_x^2$ (also known as Helmholtz operator) and its inverse is given by
$$\Lambda^{-2}u=g\ast u=\f{1}{2}\int_{\R}g(x-y)u(y)dy,$$
where $\ast$ denotes the convolution and $g(y)=e^{-|y|}/2$ is the Green function of the equation $(1-\p_x^2)u=\delta(x)$, and $\delta(x)$ is the Dirac delta distribution.

We note that $H^s(\R)$, for each $s\in\R$, is a Hilbert space when endowed with the inner product
$$\langle u,v\rangle _{H^{s}}:=\int_{\R}(1+|\xi|^2)^{s}\hat{u}(\xi)\overline{\hat{v}(\xi)}d\xi.$$

We also note that if $u\in H^1(\R)$, then its norm is given by
$\|u\|_{H^1(\R)}^2=\|u\|_{L^2(\R)}^2+\|u_x\|_{L^2(\R)}^2.$

For each $s\in\R$, $\p_x\in {\cal L}(H^{s}(\R), H^{s-1}(\R))$, where $u\mapsto u_x$, and if $s$ and $t$ are real numbers such that $s\geq t$, then ${\cal S}(\R)\subseteq H^{s}(\R)\subseteq H^t(\R)\subseteq{\cal S}'(\R)$. We, indeed, shall use the estimates $\Vert\partial_x f \Vert_{H^{s-1}(\R)}\leq \Vert f\Vert_{H^{s}(\R)}$, 
$
\Vert \Lambda^{-2}f \Vert_{H^s(\R)}\leq \Vert f\Vert_{H^{s-2}(\R)}
$
and 
$\Vert \partial_x\Lambda^{-2}f \Vert_{H^s(\R)}\leq \Vert f\Vert_{H^{s-1}(\R)}
$.

Some useful results to our purposes are:
\begin{lemma}\label{lema2.1}\textsc{[Algebra property]}
For $s>1/2$, there is a constant $c_s>0$ such that $\Vert fg \Vert_{H^s(\R)}\leq c_s\Vert f\Vert_{H^s(\R)}\Vert g\Vert_{H^s(\R)}$.
\end{lemma}

\begin{proof}
See \cite{linares} or \cite{taylor}, pages 51 and exercise 6 on page 320, respectively.
\end{proof}

\begin{lemma}\textsc{[Sobolev Embedding Theorem]}\label{lema2.2}
If $s>1/2$ and $u\in H^{s}(\R)$, then $u$ is bounded and continuous. Moreover, in case we have $s>1/2+k$, then $H^s(\R)\subseteq C^k_0(\R)$.
\end{lemma}

\begin{proof}
See \cite{linares,taylor}, pages 47 and 317, respectively.
\end{proof}

As a consequence of the Sobolev Embedding Theorem, if $s>1/2$, then $u\in L^\infty(\R)$. Moreover, if $u\in H^{s}(\R)$, with $s>1/2+k$, for a certain non-negative number $k$, then $u\in C^{k}_0$. In any case, $\|u\|_{L^\infty(\R)}\leq c_s\|u\|_{H^s(\R)}$, for some positive constant $c_s$ depending on $s$.

We recall that if $E$ and $F$ are Banach spaces, a mapping $f:E\rightarrow F$ is called Lipschitz if there exists $k>0$ such that $\|f(u)-f(v)\|_F\leq k\|u-v\|_E$, for all $u,\,v\in E$.

\begin{lemma}\label{lema2.3}
Let $F\in C^\infty(\R)$ such that $F(0)=0$. If $f\in H^s(\R)\cap L^\infty(\R)$, $s\geq 0$, then $\|F(f)\|_s\leq K\|f\|_s$, where $K$ depends only on $\|f\|_{L^\infty(\R)}$, and $\|f\|_s=\|f\|_{H^s(\R)}+\|f\|_{L^\infty(\R)}$.
\end{lemma}

\begin{proof}
See Proposition on page 1065 of \cite{const-mol}.
\end{proof}

Note that for $s>1/2$ and in view of the Sobolev Embedding Theorem, $H^s(\R)\cap L^\infty(\R)=H^s(\R)$.

We conclude this subsection with a very useful differential inequality.

\begin{lemma}\textsc{[Gronwall inequality]}\label{lema2.4}
Let $y\in C^1(I,\R)$, where $I$ is an interval on $\R$ containing $0$. Assume that $y'(t)\leq\beta(t)y(t)$, where $\beta$ is a smooth function. Then $y(t)\leq y(0)e^{\int_{0}^t\beta(s)ds}$.
\end{lemma}
\begin{proof}
See \cite{hunter}, page 56.
\end{proof}

We note that the Gronwall inequality also implies $y(t)\geq y(0)e^{\int_{0}^t\beta(s)ds}$ if $y'(t)\geq\beta(t)y(t)$.

\subsection{Proof of Theorem \ref{teo1.1}}\label{subsec2.2}

%\begin{definition}\label{def2.1}
%Let $u\in C^{0}[0,T),E)$, with $T>0$, be a solution of a differential equation, where $E$ is a suitable function space. The number $T$ is called maximal time of the existence of the solution, or lifespan, and it usually depends on the initial data of the equation, and the space $E$. We say that $u$ is a global solution if $T=\infty$. If $u$ is only defined for $T<\infty$, we then say that $u$ develop a blow-up in finite time.
%\end{definition}

Our main ingredient to prove the local-well posedness of equation \eqref{1.0.1} with initial data $u(0,x)=u_0(x)$ is the following result, due to Tosio Kato \cite{kato}.

\begin{lemma}\label{lema2.5}(Kato's theorem)
Let $A(u)$ be a linear operator and consider the problem 
\bb\label{2.3.1}
\left\{
\ba{l}
\ds{\f{d u}{dt}+A(u)u=f(u)\in X,\quad t\geq0},\\
\\
u(0)=u_0\in Y.
\ea
\right.
\ee
Assume that the following conditions are satisfied:
\begin{enumerate}
\item[{\bf C1}] Let $X$ and $Y$ be reflexive Banach spaces, such that $Y\subseteq X$ and the inclusion $Y\hookrightarrow X$ is continuous and dense. In addition, there exists an isomorphism $S:Y\rightarrow X$ such that $\|u\|_Y=\|Su\|_X$;

\item[{\bf C2}]
There exist a ball $W$ of radius $R$ such that $0\in W\subseteq Y$ and a family of operators $(A(u))_{u\in W}\subseteq {\cal L}(X)$ such that $-A(u)$ generates a $C_0$ semi-group in $X$ with $\|e^{-s A(u)}\|_{{\cal L}(X)}\leq e^{\be s}$, for any $u\in W$, $s\geq0$, for a certain real number $\be$.

\item[{\bf C3}]
Let $S$ be the isomorphism in Condition ${\bf C1}$. Then $B(u):=[S,A(u)]S^{-1}\in {\cal L}(X)$. Moreover, there exists constants $c_1$ and $c_2$ such that $\|B(u)\|_{{\cal L}(X)}\leq c_1$, $\|B(u)-B(v)\|_{{\cal L}(X)}\leq c_2\|u-v\|_{Y}$, for all $u,\, v\in W$ 

\item[{\bf C4}]
For any $w\in W$, $Y\subseteq \dom(A(w))$ and $\|A(u)-A(v)\|_{{\cal L}(Y;X)}\leq c_3\|u-v\|_{X}$, for any $u,\,v\in W$.

\item[{\bf C5}]
The function $f:X\rightarrow X$ satisfy the following conditions:
\begin{enumerate}
    \item $\left.f\right|_{W}:W\rightarrow Y$ is bounded, that is, there exists a constant $c_4$ such that $\|f(w)\|_Y\leq c_4$, for all $w\in W$;
    \item $\left.f\right|_{W}:W\rightarrow X$ is Lipschitz when taking the norm of $X$ into account, that is, there is another constant $c_5$ such that $\|f(u)-f(v)\|_{X}\leq c_5\|u-v\|_X$, for all $u,\,v\in W$.
\end{enumerate}
\end{enumerate}
If $u_0\in W$, then there is $T>0$ such that $(\ref{2.3.1})$ has a unique solution $u\in C^0([0,T),W)\cap C^1([0,T),X)$, with $u(0)=u_0$.
\end{lemma}

The constants mentioned in the conditions in the lemma above depend on the radius $R$ of $W$, see \cite{escher,kato,mustafa,blanco}.

At first sight Lemma \ref{lema2.5} does not seem to be applicable to \eqref{1.0.1}, since in \eqref{2.3.1} we have an evolution equation. This difficulty can be easily overcome by making use of the isometric isomorphism $\Lambda^{2}:H^{s}(\R)\rightarrow H^{s-2}(\R)$ and its inverse. More precisely, we have the following:

\begin{proposition}\label{prop2.1}
Equation \eqref{1.0.1} can be rewritten as
\bb\label{2.3.2}
u_t+(u+\Gamma)u_x=\Lambda^{-2}\p_x h(u)-\Lambda^{-2}\p_x\left(u^2+\f{u_x^2}{2}\right)-\lambda u,
\ee
where 
\bb\label{2.3.3}
h(u):=(\al+\Gamma)u+\f{\beta}{3}u^3+\f{\gamma}{4}u^4.
\ee
\end{proposition}
\begin{proof}
Applying the operator $\Lambda^2$ into \eqref{2.3.2} we obtain the equation in \eqref{1.0.2}, which is nothing but \eqref{1.0.1}.
\end{proof}
\begin{remark}\label{rem2.1}
Equation \eqref{2.3.2} explains why very often in the literature of CH type equations they are referred as {\it non-local evolution equations}. Firstly, note that it is now an evolution equation. However, the price to transform the non-evolution equation \eqref{1.0.1} into an evolution one is the arising of {\it non-local terms} given by the action of the operator $\Lambda^{-2}$, which is nothing but a convolution and brings non-local terms to \eqref{2.3.2}. 
\end{remark}
In view of Proposition \ref{prop2.1}, the Cauchy problem \eqref{1.0.2} is equivalent to
\bb\label{2.3.4}
\left\{
\ba{l}
\ds{u_t+(u+\Gamma)u_x=\Lambda^{-2}\p_x h(u)-\Lambda^{-2}\p_x\left(u^2+\f{u_x^2}{2}\right)-\lambda u},\quad t>0,\,\,\,x\in\R,\\
\\
u(0,x)=u_0(x),
\ea
\right.
\ee
where $h$ is given by \eqref{2.3.3}.

Let us assume that $u$ and $h(u)$ belongs to $H^{s}(\R)$, for a certain $s$. Since $H^{s}(\R)$ is continuously and densely embedded into $H^{s-1}(\R)$, then the right hand side of \eqref{2.3.2} can be viewed as a mapping $F:H^{s-1}(\R)\rightarrow H^{s-1}(\R)$, defined by
\bb\label{2.3.5}
F(u):=\Lambda^{-2}\p_x h(u)-\Lambda^{-2}\p_x\left(u^2+\f{u_x^2}{2}\right)-\lambda u,
\ee
where $h(u)$ is given by \eqref{2.3.3}. Moreover, if we define $A(u):=(u+\Gamma)\p_x$, we have a linear operator that maps a function $v\in H^s(\R)$ into $A(u)v=(u+\Gamma )v_x$. If we take $s>3/2$, $u$ and $v$ in $H^s(\R)$, then $v_x\in H^{s-1}(\R)$ and the algebra property holds, implying that $A(u)v\in H^{s-1}(\R)$. Moreover, if $s>1/2$, then the function $h$ satisfies the conditions in Lemma \ref{lema2.3}. 

We note that \eqref{2.3.4} is an equation in the variables $(t,x)$, but for each fixed $x$, it is an equation of the type \eqref{2.3.1}. We are ready to prove Theorem \ref{teo1.1}.

%These observations, jointly with the fact that now \eqref{2.3.4} is of the form \eqref{2.3.1}, {\it suggest} that the existence and uniqueness of \eqref{2.3.4} can be proved by using Kato's theorem with $X=H^{s-1}(\R)$ and $Y=H^{s}(\R)$. Our next result shows that this is, in fact, the correct way to pursue the well-posedness, at least at the local level.

{\bf Proof of Theorem \ref{teo1.1}.}
It is enough to verify that $A(u):=(u+\Gamma)\p_x$ and $F(u)$ given by \eqref{2.3.5} satisfy the conditions in Kato's theorem. From the results proved in \cite{silvajde2019,liu2011} we see that $A(u)$ satisfies all required conditions. 

Let $F(u)=f(u)+\lambda u$, with $f(u)=\Lambda^{-2}\p_x\left( h(u)-u^2-u_x^2/2\right)$. Then, for any norm $\|\cdot\|$, we have
$$
\|F(u)-F(v)\|\leq \|f(u)-f(v)\|+|\lambda|\|u-v\|.
$$
This means that $F$ satisfies the last condition in Kato's theorem if and only if $f$ does, which follows from the results proved in \cite{silvajde2019,liu2011,mustafa}, see also \cite{blanco}.
\quad\quad\quad\quad\quad\quad\quad\quad\quad\quad\quad\quad\quad\quad\quad\quad\quad\quad$\square$

\section{Time dependent conserved quantities}\label{sec3}

In this section we shall find some quantities that are conserved on the solutions of the equation in \eqref{1.0.2} for a suitable choice of the initial data. More specifically, we shall show that the quantities
$$
e^{\lambda t}\int_\R u(t,x)dx\quad\text{and}\quad e^{2\lambda t} \int_\R \left(u(t,x)^2+u_{x}(t,x)^2\right)dx
$$
are constants. This is an immediate consequence of Theorem \ref{teo1.2} that will play vital role in the investigation of the existence of global solutions and blow-up phenomena. 

The relations above show that the result of the integrals has exponential decaying and, in particular, it makes the Sobolev norm $\|\cdot\|_{H^1(\R)}$ of the solutions of \eqref{1.0.1} go to $0$ if $\lambda>0$ and $t\rightarrow\infty$, while for $\lambda<0$ it assures in a very simple, but elegant, way the presence of unbounded solutions.

{\bf Proof of Theorem \ref{teo1.2}.}
Observe that $u$ satisfies the requirements in Proposition \ref{prop1.1}. Substituting $v=u$ into \eqref{3.0.1} and \eqref{3.0.2}, noticing that $\left.E\right|_{v\equiv u}\equiv0$ and integrating, we obtain the results.
\quad\quad\quad$\square$

\begin{corollary}\label{cor3.1}
Equation \eqref{1.0.1} conserves energy if and only if $\lambda=0$. If $\lambda<0$, then both ${\cal H}_0(t)$ and ${\cal H}_1(t)$ are unbounded. If $\lambda>0$, then ${\cal H}_0(t),\,{\cal H}_1(t)\rightarrow0$ as $t\rightarrow\infty$.
\end{corollary}

We observe that in the case $\lambda\geq0$, then $H_i(t)\leq H_i(0)$, $i=1,2$. Therefore, if $\lambda>0$ and $u_0\in H^1(\R)$, equation \eqref{3.0.5} implies that the Sobolev norm of the corresponding solution of \eqref{1.0.1} satisfying $u(0,x)=u_0(x)$ is time-dependent and decreases along time if $u_0\not\equiv0$. Then, by the Sobolev Embedding Theorem, 
$\|u\|^2_{L^\infty(\R)}\leq 2 {\cal H}_1(t)\leq 2 {\cal H}_1(0)=\|u_0\|^2_{H^1(\R)}<\infty.$

Assume that $u_0\not\equiv0$ and $\lambda<0$. Theorem \ref{teo1.2} implies that if $T=\infty$, then $|{\cal H}_1(t)|\rightarrow\infty$ as $t\rightarrow\infty$, meaning that the solution of \eqref{1.0.2} is unbounded.

We summarise the comments above in the next result.

\begin{corollary}\label{cor3.2}
Assume that the initial data $u_0$ of the problem \eqref{1.0.2} belongs to $H^s(\R)$, $s>3/2$. If $u_0=0$, then the solution is globally defined and $u(t,x)\equiv0$, for any $(t,x)\in[0,\infty)\times\R$ and any value of $\lambda$. If $u_0\not\equiv0$ and $\lambda<0$, then the solution is finite on each $[0,T]$, for any $T<\infty$, but is not bounded on $[0,\infty)$, while if $\lambda>0$, then the solution $u$ is bounded from above by $\|u_0\|_{H^1(\R)}$, for any $(t,x)\in[0,T)\times\R$.
\end{corollary}

From Corollary \ref{cor3.2} we infer the presence of a necessary ingredient to the rising of wave breaking: if $\lambda>0$, then the solution $u$ is spatially bounded for $t<T$. In Section \ref{sec4} we shall retake this fact to find conditions to figure out wave breaking of the solutions of the problem \eqref{1.0.2}. 

{\bf Proof of Theorem \ref{teo1.3}.}
Consider the identity \eqref{3.0.1} with $v=u$. We can rewrite it as
$$
0=\lambda m+\f{\p m}{\p t}+\f{\p}{\p x}\left(\f{3}{2}u^2-uu_{xx}-u_x^2-\al u -\f{\be}{3}u^3-\f{\gamma}{4}u^4-\Gamma u_{xx}\right).
$$
Integrating the expression above over $\R$, we have
$$
\f{d}{dt}\int_{\R}m\,dx+\lambda \int_{\R}m\,dx=0.
$$
Solving the ODE above and taking into account that at $t=0$ we have $m(0,x)=m_0(x)$ and obtain the first equality in \eqref{3.0.6}.

The last equality is obtained in a similar way by integrating directly \eqref{3.0.1} with $v$ replaced by $u$.

It remains to be proved the middle equality, but it is a direct consequence of the fact $m=u-u_{xx}$ and $u_x\rightarrow0$ as $|x|\rightarrow\infty$.
\quad\quad\quad\quad\quad\quad\quad\quad\quad\quad\quad\quad\quad\quad\quad\quad\quad\quad\quad\quad\quad\quad\quad\quad\quad\quad\quad\quad\quad$\square$

\section{Blow-up scenario}\label{sec4}

Here we investigate the conditions for the occurrence of wave breaking of the solutions of \eqref{1.0.2}. Our main influence here is the text by Escher \cite{escher}, the works by Constantin and Escher \cite{const1998-1,const1998-2,const1998-3,const2000-1}, and Mustafa \cite{mustafa}. However, in view of the presence of the function $h(u)$ and the term $\lambda u$ in \eqref{2.3.4}, their ideas are not directly applicable to our problem. In fact, we can deal with the term $\lambda u$ following similar procedures presented in \cite{wujmp2006,wujde2009}. The main issue in our case is the term $h(u)$ in \eqref{2.3.4}, which we need to control in order to put our problem in a suitable place to be tractable. %, as it was done in \cite{silva2019} and also in \cite{mustafa}. Our strategy here is to seek the possibility of an intersection between these two venues, which at first sight are not connected, in order to have a suitable place to apply simultaneously the ideas we have mentioned.

\subsection{Preliminary results}

Here we determine {\it when} the solution of the equation possesses finite $H^3(\R)-$norm. We begin with the following result:

\begin{theorem}\label{teo4.1}
Let $u$ be a solution of \eqref{1.0.1} with initial data $u(0,x)=u_0(x)$ and $m_0=u_0(x)-u_0''(x)$. Assume that $m_0\in H^1(\R)$ and there exists a positive constant $k$ such that $u_x>-k$. Then there exist a function $\sigma\in C^{1}(\R)$ such that
$\|u\|_{H^3(\R)}\leq \sigma(t)\,\|m_0\|_{H^1}(\R)$. In particular, $u$ does not blow up in finite time.
\end{theorem}
\begin{proof}
We begin with recalling that $\Lambda^{-2}$ is a isometric isomorphism between $H^{s}$ and $H^{s+2}$. Moreover, since $m=u-u_{xx}$, then $u=\Lambda^{-2}m$. Our strategy in the present demonstration is to prove that $\|m\|_{H^1(\R)}\leq\sigma(t)\|m_0\|_{H^1(\R)}$, for a certain $\sigma\in C^1(\R)$. The result is then obtained from the relations 
$\|u\|_{H^3(\R)}=\|\Lambda^{-2}m\|_{H^3(\R)}=\|m\|_{H^1(\R)}$.

Note that 
\bb\label{4.0.6}
\f{d}{d t}\|m\|^2_{H^1(\R)}=\f{d}{d t}\|m\|^2_{L^2(\R)}+\f{d}{d t}\|m_x\|^2_{L^2(\R)}=2\left(\langle m,m_t\rangle_{L^2(\R)}+\langle m_x,m_{tx}\rangle_{L^2(\R)}\right).
\ee
Let us find the parcels of the right hand side of \eqref{4.0.6}. From \eqref{1.0.3} and \eqref{2.3.3} we have
\bb\label{4.0.7}
m_t=-(u+\Gamma)m_x-2u_xm-\lambda m+\p_xh(u).
\ee
Therefore,
\bb\label{4.0.8}
\ba{lcl}
\langle m,m_t\rangle_{L^2(\R)}&=&\ds{-\langle m,um_x\rangle_{L^2(\R)}-\Gamma\langle m,m_x\rangle_{L^2(\R)}-2\langle u_x,m^2\rangle_{L^2(\R)}}\\
\\
&&\ds{-\lambda\langle m,m\rangle_{L^2(\R)}+\langle m,\p_x h(u)\rangle_{L^2(\R)}}\\
\\
&=&\ds{-\f{3}{2}\langle u_x,m^2\rangle-\lambda\|m\|^2_{L^2(\R)}+\langle m,\p_x h(u)\rangle},
\ea
\ee
where we used the relations $\langle m,u_xm\rangle=\langle u_x,m^2\rangle$, $\langle m,um_x\rangle=\langle u,m m_x\rangle=-\langle u_x,m^2\rangle/2$ and
$$\langle m,m_x\rangle=\int_{\R}m m_x dx=\f{1}{2}\int_\R\p_{x}m^2 dx=0.$$

Deriving \eqref{4.0.7} with respect to $x$ and substituting the result into $\langle m_x,m_{tx}\rangle_{L^2(\R)}$ yield
\bb\label{4.0.9}
\ba{lcl}
\langle m_x,m_{tx}\rangle_{L^2(\R)}&=&\ds{-\langle m_x,um_{xx}\rangle_{L^2(\R)}-\Gamma m_x,m_{xx}\rangle_{L^2(\R)}-\langle m_x,u_xm_x\rangle_{L^2(\R)}}\\
\\
&&\ds{-2\langle m_x,u_{xx}m\rangle_{L^2(\R)}-2\langle m_x,u_xm_x\rangle_{L^2(\R)}}\\
\\
&&\ds{-\lambda\langle m_x,m_x\rangle_{L^2(\R)}+\langle m_x,p_x^2h(u)\rangle_{L^2(\R)}}\\
\\
&=&\ds{-\f{5}{2}\langle u_x,m_x^2\rangle_{L^2(\R)}-2\langle u_{xx},mm_x\rangle_{L^2(\R)}}\\
\\
&&\ds{-\lambda\|m_x\|^2_{L^2(\R)}+\langle m_x,\p_{x}^2h(u)\rangle_{L^2(\R)}}\\
\\
&=&\ds{-\f{5}{2}\langle u_x,m_x^2\rangle_{L^2(\R)}+\langle u_{x},m^2\rangle_{L^2(\R)}}\\
\\
&&\ds{-\lambda\|m_x\|^2_{L^2(\R)}+\langle m_x,\p_{x}^2}h(u)\rangle_{L^2(\R)},
\ea
\ee
where we used $\langle u_{xx},mm_x\rangle_{L^2(\R)}=\langle u,\p_{x} m^2\rangle_{L^2(\R)}/2-\langle m,mm_x\rangle_{L^2(\R)}=-\langle u_x,m^2\rangle_{L^2(\R)}/2$.

From \eqref{4.0.8}, \eqref{4.0.9} and after some manipulation, we have
\bb\label{4.0.10}
\ba{lcl}
\ds{\f{d}{d t}\|m\|^2_{H^1(\R)}}&=&\ds{-\langle u_x,m^2\rangle_{L^2(\R)}-5\langle u_x,m_x^2\rangle_{L^2(\R)}-2\lambda \|m\|_{H^1(\R)}^2}\\
\\
&&\ds{+2\left(\langle m,\p_x^2 h(u)\rangle_{L^2(\R)}+\langle m_x,\p_x^2 h(u)\rangle_{L^2(\R)}\right)}.
\ea
\ee

Let $I:=\langle m,\p_x^2 h(u)\rangle_{L^2(\R)}+\langle m_x,\p_x^2 h(u)\rangle_{L^2(\R)}$. Then
$$
\ba{lcl}
I&=&\langle m,\p_x^2 h(u)\rangle_{L^2(\R)}+\langle m_x,\p_x^2 h(u)\rangle_{L^2(\R)}\\
\\
&=&\langle m,\p_{x}^2h(u)-\p_x^3h(u)\rangle_{L^2(\R)}=\langle m,\Lambda^2\p_x h(u)\rangle_{L^2(\R)}\leq\|m\|_{L^2(\R)}\|\Lambda^2\p_x h(u)\|_{L^2(\R)}.
\ea
$$
Since $2ab\leq a^2+b^2$, for any real numbers $a$ and $b$, we have
$$
I\leq \f{\|\Lambda^2\p_x h(u)\|^2_{L^2(\R)}+\|m\|_{L^2(\R)}^2}{2}.
$$
We still have the inequality $\|\Lambda^2\p_x h(u)\|_{L^2(\R)}\leq\|\p_x h(u)\|_{H^2(\R)}\leq\|h(u)\|_{H^3(\R)}$. 

{\bf Claim.} There exists a constant $c_1$ depending only on $\|u_0\|_{H^1(\R)}$ such that $\|h(u)\|_{H^3(\R)}\leq c_1\|u\|_{H^3(\R)}$.

We note that if our claim is true, then $\|h(u)\|_{H^3(\R)}\leq c_1\|u\|_{H^3(\R)}=c_1\|\Lambda^{-2}u\|_{H^1(\R)}=c_1\|m\|_{H^1(\R)}$
and since $\|m\|_{L^2(\R)}\leq \|m\|_{H^1(\R)}$ we conclude that
\bb\label{4.0.11}
\langle m,\p_x^2 h(u)\rangle_{L^2(\R)}+\langle m_x,\p_x^2 h(u)\rangle_{L^2(\R)}\leq c\|m\|^2_{H^1(\R)},
\ee
for some positive constant $c$. In addition, we have
\bb\label{4.0.12}
\ba{lcl}
\ds{-\langle u_x,m^2\rangle_{L^2(\R)}-5\langle u_x,m_x^2\rangle_{L^2(\R)}}&=&\ds{\int_{\R}(-u_x)(m^2+5m_x^2)dx\leq k\int_{\R}(m^2+5m_x^2)dx}\\
\\
&\leq& 5k\|m\|_{H^1(\R)}^2.
\ea
\ee
Substitution of \eqref{4.0.11} and \eqref{4.0.12} into \eqref{4.0.10} reads
$$
\f{d}{d t}\|m\|^2_{H^1(\R)}\leq (5k+c+2\lambda)\|m\|_{H^1(\R)}^2.
$$
Using the Gronwall inequality, we conclude that
$
\|m\|_{H^1(R)}^2\leq e^{(5k+c+2\lambda) t} \|m_0\|^2_{H^1(\R)}=:\sigma(t)^2\|m_0\|^2_{H^1(\R)},
$
which is sufficient to have the result proved.

To conclude the proof, we must now prove the claim. To do it, note that
$$
h(u)=\int_0^1uh'(su)ds.
$$
From the last equality we easily conclude that $|h(u)|\leq |u|\sup\{|h'(r)|,\,|r|\leq \|u\|_{L^\infty(\R)}\}$. Moreover, we also know that $\|u\|_{L^\infty(\R)}\leq\|u\|_{H^1(\R)}$ (see, for instance, \cite{mustafa}, page 1396). This, combined with the fact that $\|u\|_{H^1(\R)}\leq \|u_0\|_{H^1(\R)}$ implies that $|h(u)|\leq |u|\sup\{|h'(r)|,\,|r|\leq \|u_0\|_{H^1(\R)}\}$. Let $c_1:=\sup\{|h'(r)|,\,|r|\leq \|u_0\|_{H^1(\R)}\}$. Then $|h(u)|\leq c_1|u|$, which yields the desired result.
\end{proof}

Observe that $H^3(\R)\subseteq H^s(\R)$, for any $s\leq 3$. Moreover, this embedding is dense and continuous. This proves the following

\begin{corollary}\label{cor4.2}
Let $u_0\in H^s$, $s\geq3/2$, and $u$ be the corresponding solution to \eqref{1.0.1} with initial data $u(0,x)=u_0(x)$. Assume that $u_x>-k$, for some positive constant $k$. Then $\|u\|_{H^s(\R)}\leq \sigma(t)\|u_0\|_{H^s(\R)}$, for a certain positive function $\sigma\in C^1(\R)$.
\end{corollary}

\subsection{Wave breaking criteria}

We recall that wave breaking occurs when
$$\sup_{(t,x)\in[0,T)\times\R}|u(t,x)|<
\infty,\quad\text{and}\quad\lim\sup\limits_{t\rightarrow T}\left(\sup_{x\in\R}|u_x(t,x)|\right)=\infty.$$

Corollary \ref{cor3.1} implies that if $\lambda>0$, then $u$ is bounded, which is a necessary condition for the appearance of wave breaking. On the other hand, Theorem \ref{teo4.1} assures that if $u$ is a solution of \eqref{1.0.2} and $u_x$ is bounded from below, then we do not have wave breaking. The consequence of these facts is: If we want to investigate wave breaking, then we must look for solutions having $u_x$ with no lower bound. We observe that if $u_x(t,x)<0$, then we should replace $\sup$ and $\infty$ by $\inf$ and $-\infty$, respectively, in the conditions above whenever we replace $|u_x(t,x)|$ by $u_x(t,x)$. To understand what we are bound to do, let $y(t):=\inf\{u_x(t,x),\,x\in\R\}$, which we simply write as $y(t):=\inf\limits_{x}u_x(t,x)$ from now on. Then, the wave breaking will occur if
$$-\infty=\lim\inf_{t\rightarrow T}\left(\inf_{x\in\R}u_x(t,x)\right)=\lim_{t\rightarrow T}\inf y(t)=\lim_{t\rightarrow T}\inf_{t\leq s}y(s).$$

\begin{lemma}\label{lema4.1}
Let $T>0$ and $v\in C^1([0,T),H^2(\R))$ be a given function. Then, for any $t\in[0,T)$, there exists at least one point $\xi(t)\in\R$ such that
\bb\label{5.0.1}
y(t)=\inf_{x\in\R}{v_x(t,x)}=v_x(t,\xi(t))
\ee
and the function $y$ is almost everywhere differentiable (a.e) in $(0,T)$, with $y'(t)=v_{tx}(t,\xi(t))$ {\it a.e.} on $(0,T)$.
\end{lemma}
\begin{proof}
See Theorem 2.1 in \cite{const1998-2} or Theorem 5 in \cite{escher}.
\end{proof}

Let us consider equation \eqref{2.3.2}, with $h$ given by \eqref{2.3.3}. Differentiating \eqref{2.3.2} with respect to $x$, we obtain
\bb\label{5.0.2}
u_{tx}+\f{u_x^2}{2}+(u+\Gamma)u_{xx}+\lambda u_x=u^2-\Lambda^{-2}\left(u^2+\f{u_x^2}{2}\right)-h(u)-\Lambda^{-2}h(u),
\ee
were we used the identity $\p_x^2\Lambda^{-2}=\Lambda^{-2}-1$.

We observe that \eqref{5.0.2} holds to any $(t,x)$ where $u$ is defined. If $u_0\in H^3(\R)$, then $u(t,\cdot)\in C^1(\R)$ and satisfies the conditions required in Lemma \ref{lema4.1}. Defining 
$$%\bb\label{5.0.3}
y(t)=\inf_{x\in\R}u_x(t,x)=u_x(t,\xi(t)),
$$%\ee
evaluating equation \eqref{5.0.2} at $(t,\xi(t))$ and noticing that $u_{xx}(t,\xi(t))=0$, we arrive at the following ordinary differential equation to $y$:
\bb\label{5.0.4}
y'(t)+\f{y(t)^2}{2}+\lambda y(t)=u(t)^2-F(u(t))-G(u(t)),
\ee
where $u(t):=u(t,\xi(t))$, 
\bb\label{5.0.5}
F(u):=\Lambda^{-2}\left(u^2+\f{u_x^2}{2}\right)
\ee
and
\bb\label{5.0.6}
G(u):=h(u)+\Lambda^{-2}h(u).
\ee

\begin{proposition}\label{prop4.1}
Let $F$ and $G$ be given by \eqref{5.0.5} and \eqref{5.0.6}, respectively. If $u\in H^1(\R)$, then $F(u(t))\geq u(t)^2/2$ and $|G(u(t))|\leq 3\|h(u)\|_{H^1(\R)}$.
\end{proposition}

\begin{proof}
The proof that $F(u(t))\geq u^2(t)/2$ can be found in \cite{escher}, pages 106 and 107 and, therefore, is omitted. Let us estimate $|G(u(t))|$. We first note that $\|h(u)\|_{L^\infty(\R)}\leq \|h(u)\|_{H^1(\R)}$, $h(0)=0$ and by Lemma \ref{lema2.3}, if $u\in H^1(\R)$, then $h(u)\in H^1(\R)$. Therefore,
$$
|\Lambda^{-2}h(u(t))|=\left|\int_{-\infty}^\infty e^{-|\xi(t)-y|}h(u(t,y))dy\right|\leq\int^\infty_{-\infty}e^{-|\xi(t)-y|}|h(u(t,y))|dy\leq 2\|h(u)\|_{H^1(\R)}.
$$
It is then easy to find the upper bound $|G(u(t))|\leq 3\|h(u(t))\|_{H^1(\R)}$.
\end{proof}

Let 
$$U_0:=\max\{\|u_0\|_{H^1(\R)},\|u_0\|_{H^1(\R)}^4\}.$$ Then $\|h(u)\|_{H^1(\R)}\leq\kappa U_0$, for a certain positive constant $\kappa$ depending on $\al,\,\be,\,\gamma$ and $\Gamma$. Moreover, 
$$
\ba{lcl}
2u^2(t)&=&2\ds{\left(\int_{-\infty}^{\xi(t)}u(t,y)u_x(t,y)dy-\int_{\xi(t)}^\infty u(t,y)u_x(t,y)dy\right)}\\
\\
&\leq&\ds{\int_{-\infty}^{\xi(t)}(u^2(t,y)+u_x^2(t,y))dy+\int^{\infty}_{\xi(t)}(u^2(t,y)+u_x(t,y)^2)dy=\|u(t)\|^2_{H^1(\R)}}.
\ea$$
From the comments above, Proposition \ref{prop4.1} and equation \eqref{5.0.4} we have the following inequality
\bb\label{5.0.7}
y'(t)+\f{y(t)^2}{2}+\lambda y(t)\leq \f{u(t)^2}{2}+|G(u(t))|\leq\f{1}{4}\|u_0\|^2_{H^1(\R)}+3\kappa U_0.
\ee

Suppose that $u_0\in H^3(\R)$ be an initial data of \eqref{1.0.2} such that 
\bb\label{5.0.8}
\theta u'_0(x_0)<\min\{-\|u_0\|_{H^1(\R)}^{1/2},-\|u_0\|_{H^1(\R)}^2\},
\ee
for some constant $\theta>0$ and some point $x_0\in\R$. Since $y(t)=\inf\limits_{x} u_x(t,x)$, then 
$y(0)\leq u_x(0,x_0)=u_0'(x_0)$
and
$$
\theta y(0)\leq\theta u_0'(x_0)<\min\{-\|u_0\|_{H^1(\R)}^{1/2},-\|u_0\|_{H^1(\R)}^2\},
$$
which implies
$$
U_0<\theta^2u_0'(x_0)^2\leq\theta^2y(0)^2.
$$
Suppose that we might be able to find $\epsilon$ such that 
$$
U_0\leq(1-\epsilon)\theta^2 u_0'(x_0)^2\leq(1-\epsilon)\theta^2y(0)^2,
$$
and noting that $\|u_0\|_{H^1(\R)}^2\leq U_0$, we obtain the upper bound 
$$
\f{1}{4}\|u_0\|^2_{H^1(\R)}+3\kappa U_0\leq \left(\f{1}{4}+3\kappa\right)(1-\epsilon)\theta^2y(0)^2.
$$
If we could choose $\theta\leq \sqrt{2/(1+12\kappa)}$, then the inequality \eqref{5.0.7} would read
\bb\label{5.0.9}
y'(t)+\f{y(t)^2}{2}+\lambda y(t)\leq \f{1-\epsilon}{2}y(0)^2.
\ee

We are bound to find conditions for the occurrence of wave breaking. We only need two very technical observations that will make easier the proof of Theorem \ref{teo1.4}.

\begin{remark}\label{rem4.1}
Let $u_0\in H^3(\R)$ be an initial data to the problem \eqref{1.0.2}, $y(0):=\inf\limits_{x} u_x(0,x)$. If $u_0$ satisfy the condition \eqref{5.0.8} for some $\theta>0$, let us define 
\bb\label{5.0.11}
\epsilon_0:=\f{\theta^2 u_0'(x_0)^2-\max\{\|u_0\|_{H^1(\R)},\|u_0\|_{H^1(\R)}^4\}}{\theta^2 u_0'(x_0)^2}.
\ee
Clearly $\epsilon_0\in(0,1)$ and if $\epsilon\in(0,\epsilon_0]$, then $\max\{\|u_0\|_{H^1(\R)},\|u_0\|_{H^1(\R)}^4\}\leq(1-\epsilon)\theta^2 u_0'(x_0)^2$. 
\end{remark}

\begin{remark}\label{rem4.2}
Under the same conditions in Remark \ref{rem4.1}, let
\bb\label{5.0.12}
\lambda_0:=-\f{y(0)}{4}\epsilon_0,
\ee
where $\epsilon_0$ is given by \eqref{5.0.11}. If $\lambda\in(0,\lambda_0)$, then 
    \bb\label{5.0.10}
    \f{\epsilon_0}{4\lambda}+\f{1}{y(0)}>0.
    \ee
\end{remark}

%\begin{proposition}\label{prop4.2}
%Let $u_0\in H^3(\R)$ be an initial data to the problem \eqref{1.0.2}, $y(0):=\inf\limits_{x} u_x(0,x)$. If $u_0$ satisfy the condition \eqref{5.0.8} for some $\theta$, then there exist $\epsilon_0=\epsilon_0(\theta,u_0)>0$ and $\lambda_0=\lambda_0(\theta,u_0)>0$ such that
%\begin{enumerate}
%    \item $\epsilon_0\in(0,1)$ and if $\epsilon\in(0,\epsilon_0]$, then $\max\{\|u_0\|^2_{H^1(\R)},U_0\}\leq(1-\epsilon)\theta^2 u_0'(x_0)^2$.
%    \item If $\lambda\in(0,\lambda_0)$, then 
%    \bb\label{5.0.10}
%    \f{\epsilon_0}{4\lambda}+\f{1}{y(0)}>0.
%    \ee
%\end{enumerate}
%\end{proposition}

%\begin{proof}
%A straightforward calculation shows that 
%\bb\label{5.0.11}
%\epsilon_0:=\f{\theta^2 u_0'(x_0)^2-\max\{\|u_0\|_{H^1(\R)},\|u_0\|^2_{H^1(\R)},\|u_0\|_{H^1(\R)}^3,\|u_0\|_{H^1(\R)}^4\}}{\theta^2 u_0'(x_0)^2}
%\ee
%satisfies the required conditions. For $\lambda$, note firstly that $y(0)<0$ and if \eqref{5.0.10} holds, then $\lambda <\lambda_0$, where
%\bb\label{5.0.12}
%\lambda_0:=-\f{y(0)}{4}\epsilon_0
%\ee
%and $\epsilon_0$ is given by \eqref{5.0.11}.
%\end{proof}

{\bf Proof of Theorem \ref{teo1.4}.}
In view of Corollary \ref{cor4.2} we only need to show that $u_x$ does not have any lower bound.

Under the conditions of the theorem, we note that $y(t)=\inf\limits_{x}u_x(t,x)$ satisfies the inequality \eqref{5.0.9}.

It follows from \cite{escher}, page 108, that $y(t)^2>(1-\epsilon/2)y(0)^2$. This implies that \eqref{5.0.9} can be rewritten as $y'(t)+\lambda y(t)<-\epsilon y(0)^2/4$. Moreover, it is also immediate that $y(0)^2<2y(t)^2/(2-\epsilon)<2y(t)^2$. Taking all of these inequalities into account, substituting them into \eqref{5.0.9} and taking $\epsilon=\epsilon_0/2$, where $\epsilon_0$ is given in Remark \ref{rem4.1}, we obtain
$$
y'(t)+\lambda y(t)\leq-\f{\epsilon_0}{4}y(t)^2.
$$
We observe that the last inequality shows that $y(\cdot)$ is a decreasing function and since $y(0)<0$, then $y(t)<0$ for any $t>0$. Moreover,
$$
\f{d}{dt}\left(\f{1}{e^{\lambda t}y(t)}\right)=-e^{\lambda t}\f{y'+\lambda y}{y^2}\geq\f{\epsilon_0}{4}e^{-\lambda t}
$$ 
and a direct integration followed by simple manipulation yield
$$
e^{\lambda t}\left(\f{\epsilon_0}{4\lambda}+\f{1}{y(0)}\right)\leq\f{\epsilon_0}{4\lambda}+\f{1}{y(t)}\leq\f{\epsilon_0}{4\lambda}.$$

Due to $\lambda<\lambda_0$, we conclude that
$$%\bb\label{5.0.14}
0<e^{\lambda t}\left(\f{\epsilon_0}{4\lambda}+\f{1}{y(0)}\right)\leq\f{\epsilon_0}{4\lambda},
$$%\ee
which forces $t$ to be finite. Therefore, the solution $u$ cannot be defined for all values of $t$, and we then conclude the existence of a finite lifespan $T>0$. Since $m_0\in H^1(\R)$, Theorem \ref{teo4.1} implies that $u_x$ has no lower bound, that is, $\lim\inf\limits_{t\rightarrow T}\left(\inf\limits_{x\in\R}u_x(t,x)\right)=-\infty$.\quad\quad\quad\quad\quad\quad\quad\quad\quad\quad\quad\quad\quad\quad$\square$

\section{Comments on the limitations to ensure global existence of solutions}\label{sec5}

In this section we show the limitations to completely describe the scenario for global existence. 
We begin with a technical result proving the existence of a diffeomorphism.
\begin{proposition}\label{prop4.1}
Let $u\in C^{1}([0,T),H^2(\R))$ be a solution of \eqref{1.0.1}. Then the problem
\bb\label{4.0.1}
\left\{
\ba{l}
q_t(t,x)=u(t,q)+\Gamma,\\
\\
q(0,x)=x
\ea
\right.
\ee
has a unique solution $q(t,x)$ and $q_x(t,x)>0$, for any $(t,x)\in[0,T)\times\R$. Moreover, $q(t,\cdot)$ is an increasing diffeomorphism of the line.
\end{proposition}
\begin{proof}
We prove that the system has a unique solution $q(t,x)$ and that $q_x(t,x)>0$, for any $(t,x)\in[0,T)\times\R$. The proof that this function is a diffeomorphism is the same as for the CH equation and it can be found in \cite{const2000-1}, Theorem 3.1.

Since $u\in C^1([0,T)H^2(\R))$ and $H^2(\R)\subset C^1(\R)$, then $u\in C^{1}([0,T)\times\R,\R)$, which means that problem \eqref{4.0.1} has a unique solution. If we differentiate \eqref{4.0.1} with respect to $x$, and noticing that $\p_x u(t,q)=u_{x}q_x$, we have another IVP, given by
\bb\label{4.0.2}
\left\{
\ba{l}
\ds{\p_tq_x(t,x)=u_x(t,q)q_x(t,y)},\\
\\
q_x(0,x)=1
\ea
\right.
\ee
The solution of \eqref{4.0.2} is
$$
q_x(t,x)=\exp\left(\int_0^tu_x(s,q(s,x))ds\right),
$$
which completes the proof.
\end{proof}

We note that Theorem \ref{teo4.1} describes a sufficiency condition for the global existence of solutions, but it does not give us any information about whether this condition is satisfied.

\begin{theorem}\label{teo5.1}
Let $u_0\in H^s(\R)$, $s>3/2$, and $m_0(x)=u_0(x)-u_0''(x)$. Assume that:
\begin{enumerate}
    \item there exists a point $x_0\in\R$ such that $m_0(x)\leq0$, if $x\in(-\infty,x_0]$, and $m_0(x)\geq0$, if $x\in[x_0,\infty)$;
    \item $\sign{(m_0)}=\sign{(m)}$.
\end{enumerate}
Then the solution $u$ of \eqref{1.0.1} possesses bounded from below $x-$derivative.
\end{theorem}

\begin{proof}
Since $u=p\ast m$, where $\ast$ denotes the convolution and $p=e^{-|x|}/2$, we have
$$
u(t,x)=\f{1}{2}\int_\R e^{-|x-\xi|}m(t,\xi)d\xi=\f{1}{2}e^{-x}\int_{-\infty}^xe^\xi m(t,\xi)d\xi+\f{1}{2}e^{x}\int^{\infty}_xe^{-\xi} m(t,\xi)d\xi.
$$
Differentiating this representation of $u$ with respect to $x$ gives
$$
u_x(t,x)=-\f{1}{2}e^{-x}\int_{-\infty}^xe^\xi m(t,\xi)d\xi+\f{1}{2}e^{x}\int^{\infty}_xe^{-\xi} m(t,\xi)d\xi.
$$
Since $\sign{m}=\sign{m_0}$, then $m(t,x)\leq 0$ if $ x\leq q(t,x_0)$ and $m(t,x)\geq0$ if $x\geq q(t,x_0)$, where $q$ is the function in Proposition \ref{prop4.1}. Therefore,
$$
\ba{lcl}
u_x(t,x)&=&\ds{-\f{1}{2}e^{-x}\int_{-\infty}^x e^\xi m(t,\xi)d\xi-\f{1}{2}e^{x}\int_x^\infty e^{-\xi}m(t,\xi)d\xi}\\
\\
&&\ds{+\f{1}{2}e^{x}\int_x^\infty e^{-\xi}m(t,\xi)d\xi+\f{1}{2}e^{x}\int^{\infty}_xe^{-\xi} m(t,\xi)d\xi}\\
\\
&=&\ds{-u(t,x)+e^{x}\int^{\infty}_xe^{-\xi} m(t,\xi)d\xi}.
\ea
$$
As a consequence, if $x\geq q(t,x_0)$, then $u_x(t,x)\geq -u(t,x)$. On the other hand, a similar calculation reads
$$
u_x(t,x)=u(t,x)-e^{-x}\int_{-\infty}^xe^{\xi} m(t,\xi)d\xi,
$$
which implies that $u_x(t,x)\geq u$, provided that $x\leq q(t,x_0)$.

These two facts are enough to assure that $u_x(t,x)\geq -\|u\|_{L^\infty(\R)}$. Since $\|u\|_{L^\infty(\R)}\leq \|u\|_{H^1(\R)}\leq \|u_0\|_{H^1(\R)}$, we conclude that $- \|u_0\|_{H^1(\R)}\leq u_x(t,x)$.
\end{proof}

We observe that if the conditions of Theorem \ref{teo5.1} are satisfied and if $m_0\in H^1(\R)$, then $u$ does not blow up in finite time in view of Theorem \ref{teo4.1}.

\begin{theorem}\label{teo5.2}
Let $u$ be a solution of \eqref{1.0.1} with initial data $u(0,x)=u_{0}(x)$, $m_0(x):=u_0-u_0''$. Then
\bb\label{4.0.4}
m(t,q(t,x))\,q_x^2(t,x)=m_0(x)e^{-\lambda t}+\int_{0}^t e^{\lambda(s-t)}\,q_x^2(s,x)\p_x h(u(s,q))ds,
\ee
where $q$ is the solution of the problem \eqref{4.0.1} and
$h(u)$ is the function given in \eqref{2.3.3}.
\end{theorem}

\begin{proof}

Differentiating $m(t,q(t,x)) q_x^2(t,x)$ with respect to $t$ and using \eqref{4.0.7}, we conclude that 
\bb\label{4.0.5}
\f{d}{dt}(m q_x^2)=\left[m_t+(u+\Gamma)m_x+2u_x m\right]q_x^2=-\lambda m q_x^2+\p_xh(u)q_x^2,
\ee
which is a linear ODE to $m q_x^2$. Integrating \eqref{4.0.5} and taking \eqref{4.0.1} and \eqref{4.0.2} into account we conclude that its solution is \eqref{4.0.4}. 
\end{proof}

\begin{remark}\label{rem5.1}
In Theorem \ref{teo5.2} we observe how the presence of the function $h(u)$ affects the investigation of global existence of solutions to the Cauchy problem \eqref{1.0.2}. If $h(u)=0$, from \eqref{4.0.4} we would then immediately conclude that $\sign{(m)}=\sign{(m_0)}$ and the second condition in Theorem \ref{teo5.1} would be a consequence of the first.
\end{remark}

\begin{remark}\label{rem5.2} From Theorem \ref{teo5.1} we see that the condition $\sign{(m_0)}=\sign{(m)}$ is essential, at least in the venue we followed, to prove the global existence of solutions. However, we are unable to describe completely whether such condition is satisfied, which is an open problem in the study global solvability of the equation \eqref{1.0.2}. 

A possible direction to have a complete description of the existence of global solutions is the following: Let $h(u)$ be the function given by \eqref{2.3.3}. From \eqref{4.0.10} we have
\bb\label{6.1}
m(t,q(t,x))\,q_x^2(t,x)=e^{-\lambda t}\left(m_0(x)+\int_{0}^t e^{\lambda\,s}\,q_x^2(s,x)\p_x h(u(s,q))ds\right).
\ee

If we could determine a function $\xi=\xi(x)\geq0$ such that 
$$\left|\int_{0}^t e^{\lambda(s)}\,q_x^2(s,x)\p_x h(u(s,q))ds\right|\leq \xi(x)$$
and if $|m_0(x)|\geq\xi(x)$, it would then follow from \eqref{6.1} that $\sign{(m)}=\sign{(m_0)}$. Unfortunately we have not succeed to find such a function for arbitrary values of the parameters $\al,\,\be,\,\gamma,$ and $\Gamma$ in \eqref{2.3.3}. 
\end{remark}

In line with the remarks above, we have the following corollary.
\begin{corollary}\label{cor5.1}
If $h(u)\equiv0$, $u_0\in H^s(\R)$, $s>3/2$, and the first condition in Theorem \ref{teo5.1} holds, then the solutions of the Cauchy problem \eqref{1.0.2} exists globally.
\end{corollary}
\begin{proof}
It follows from \eqref{6.1} that $\sign{(m)}=\sign{(m_0)}$ and the conclusion is a consequence of theorems \ref{teo4.1} and \ref{teo5.1}.
\end{proof}

Corollary \ref{cor5.1} recover global existence results proved in \cite{novjmp2013,novjde2016,wujmp2006,wujde2009,zhangjmp2015} regarding weakly dissipative Camassa-Holm and Dullin-Gottwald-Holm equations.

\section{Discussion}\label{sec6}

Very recently \cite{chines-adv,chines-arxiv,gui-jnl,chines-jde}, an equation of the type \eqref{1.0.1} was deduced as a model for shallow water waves with Coriolis effect. The mentioned equation has its coefficients depending on physical parameters related to the rotation of the Earth. On the other hand, in \cite{silvajde2019} we considered \eqref{1.0.1} with $\lambda=0$ and we investigated it from a complementary point of view (taking the results in \cite{chines-adv,chines-arxiv,gui-jnl,chines-jde} into account). This equation revealed to be mathematically very rich, as one can see by the multitude of travelling waves \cite{silvajde2019,gui-jnl} it possesses, local well-posedness \cite{chines-adv,silvajde2019} and wave breaking \cite{silva2019}%e investigated the global existence of solutions of \eqref{1.0.3} and we conclude that it may be globally well posed if $u_x$ is bounded from below. We also determined conditions for the appearance of wave breaking. These results follow the same strategy we used in the present work.

The main difference of our results and those established in \cite{silvajde2019,silva2019} is the presence of the term $\lambda m$ in \eqref{1.0.1} or $\lambda u$ in \eqref{2.3.4}. From a different perspective, we can also argue that the main difference between the problem we treated here with those treated in \cite{niujmaa2011,novjmp2013,novjde2016,zhangjmp2015,wujmp2006,wujde2009} is the presence of the cubic and quartic nonlinearities in \eqref{1.0.1}.

In some parts of the text we mentioned {\it dissipation}. Let us explain the term and why this is our case. Let $u_0$ be an initial data of \eqref{1.0.2} such that $0\neq u_0\in H^s(\R)$, with $s>3/2$, and $\lambda>0$. From Theorem \ref{teo1.2} we have the energy
$
{\cal H}_1(t)=e^{-2\lambda t}{\cal H}_1(0),
$
and ${\cal H}_1(0)=\|u_0\|_{H^1(\R)}^2/2>0$. Then we can easily infer
\bb\label{6.0.2}
\f{d {\cal H}_1(t)}{dt}=-2\lambda {\cal H}_1(t)=-2\lambda e^{-2\lambda t}{\cal H}_1(0)<0.
\ee

From \eqref{6.0.2} we observe that ${\cal H}_1(t)$ decreases with time, which means that the energy is not conserved along time, or better, we have loss, or {\it dissipation}, of energy. Moreover, we observe that the energy of the solution is a monotonic decreasing function of $t$.

The presence of the function $h(u)$ in \eqref{2.3.4} or, more precisely, the cubic and quartic nonlinearities in \eqref{1.0.1}, brings some complexity to the problem when compared with similar results of CH and DGH equations. For example, the condition we found for the existence of wave breaking are affected by the values of the parameters, as one can see by the range of $\theta$ in the condition \eqref{5.0.8} given by Theorem \ref{teo1.4}. Moreover, the value of $\epsilon_0$ also depends on powers of the norm $\|u_0\|_{H^1(\R)}$, see \eqref{5.0.11}, as well as the upper bound to $\lambda$, as shown in \eqref{5.0.12}. If $\|u_0\|_{H^1(\R)}\leq1$ then \eqref{5.0.11} and \eqref{5.0.12} reduce to
$$
\epsilon_0=\f{\theta^2 u'_0(x_0)^2-\|u_0\|_{H^1(\R)}}{\theta^2 u_0'(x_0)^2}\quad\text{and}\quad \lambda_0=-\f{y(0)}{4}\f{\theta^2 u'_0(x_0)^2-\|u_0\|_{H^1(\R)}}{\theta^2 u_0'(x_0)^2},
$$
where $y(0)=\inf\limits_{x}u_x(0,x)$ and the possible values for $\theta$ are given in Theorem \ref{teo1.4}. However, in case $\|u_0\|_{H^1(\R)}>1$ and $\gamma\neq0$, then these constants changes to
$$
\epsilon_0=\f{\theta^2 u'_0(x_0)^2-\|u_0\|^4_{H^1(\R)}}{\theta^2 u_0'(x_0)^2}\quad\text{and}\quad \lambda_0=-\f{y(0)}{4}\f{\theta^2 u'_0(x_0)^2-\|u_0\|^4_{H^1(\R)}}{\theta^2 u_0'(x_0)^2},
$$
evidencing how the higher order nonlinearities affects these parameters and the wave breaking as well.

The wave breaking phenomena of the solutions of \eqref{1.0.2} is assured by Theorem \ref{teo1.4} provided that $\lambda\in(0,\lambda_0)$. We would like to point out the following comments about the parameter $\lambda$:
\begin{enumerate}
    \item The presence of the term $\lambda\, m$ in \eqref{1.0.2}, with $\lambda>0$, is enough to guarantee the existence of time-decaying solutions to the equation with sufficient regularity, no matter the value of $\lambda$, as can infer from equation \eqref{6.0.2}. In particular, larger values of $\lambda$ imply a faster decaying of energy than the small ones.
    \item Although larger values of $\lambda$ result into fast decaying of the energy and, consequently, the fast vanishing of the solutions, they cannot guarantee the existence of wave breaking of the solutions. Actually, our results, namely Theorem \ref{teo1.4}, can only assure the appearance of wave breaking under restrictive conditions, which among then, small values of $\lambda$. By small values of $\lambda$ we mean those smaller than $\lambda_0$, see \eqref{5.0.12}.
    \item Theorem \ref{teo1.4} does not give any information to us if wave breaking may or not occur if $\lambda\geq\lambda_0$. Actually, this is an open question.
    \item In line with the previous comments, only very small values of $\lambda$ surely allow the wave breaking of solutions of \eqref{1.0.2}. We observe that for small values of $\lambda$, the term $\lambda\,m$ in \eqref{1.0.2} can be interpreted as a perturbation in the equation in \eqref{1.0.3}.  This perturbation, no matter how small it is, is enough to bring the vanishing of the solutions with enough regularity of the equation.
    \item The conditions for wave breaking phenomena for equation \eqref{1.0.3} were investigated in \cite{silva2019}. We note that if $\lambda$ is small (in the sense mentioned above), the conditions for wave breaking of \eqref{1.0.2} are essentially unaltered when compared to those for \eqref{1.0.3}, see \cite{silva2019}.
\end{enumerate}
\section{Conclusion}\label{sec7}

In the present work we investigated equation \eqref{1.0.1} and, with more emphasis, its corresponding Cauchy problem \eqref{1.0.2}. 

Our main results can be summarised as follows:
\begin{enumerate}
    \item We prove the local well posedness to the Cauchy problem \eqref{1.0.2}, see Theorem \ref{teo1.1}.
    \item We established conservation laws and conserved quantities for the equation and the problem, see Theorem \ref{teo1.2}. In particular, we proved that the solutions of \eqref{1.0.2} are bounded from above by the Sobolev norm of the initial data.
    \item We also obtain sufficient condition for the appearance of wave breaking of the solutions of \eqref{1.0.2}, see Theorem \ref{teo1.4}.
\end{enumerate}

\section*{Acknowledgements}

The idea of this paper occurred during the period the  I was as a visiting professor at Silesian University in Opava, Czech Republic, where this research begun. I would like to express my deeply gratitude to the Mathematical Institute of the Silesian University in Opava for the warm hospitality and very nice work atmosphere I found there. Particular thanks are given to Professor A. Sergyeyev and Professor R. Popovych for the stimulating discussions we had during my visit. I would also like to thank Dr. P. L. da Silva for all discussions regarding the subject of this paper and for her firm and continuous encouragement. Last, but not least, I want to express my deepest gratitude to Prof. A. Sergyeyev, who was a perfect host during my stay in Opava.

I am also thankful to CNPq (grants 308516/2016-8 and 404912/2016-8) for financial support.


\begin{thebibliography}{10}

\bibitem{anco} S. Anco, P. L. da Silva and I. L. Freire, \newblock A family of wave-breaking equations generalizing the Camassa-Holm and Novikov equations, \newblock \emph{J. Math. Phys.}, vol. 56, paper 091506, (2015).


\bibitem{brezis} H. Brezis, Functional Analysis, Sobolev Spaces and Partial Differential Equations, Springer, (2011).


\bibitem{boz}Y. Bozhkov, I. L. Freire and N. Ibragimov, Group analysis of the Novikov equation, Comp. Appl. Math., v 33, 193--202, 2014. 


\bibitem{chprl} R. Camassa, D.D. Holm, An integrable shallow water equation with peaked solitons, Phys. Rev. Lett., vol. 71, 1661--1664, (1993).



\bibitem{chines-adv} R. M. Chen, G. Gui and Y. Liu, On a shallow-water approximation to the Green--Naghdi equations with the Coriolis effect, Adv. Math., vol. 340, 106--137, (2018).  



\bibitem{const1998-1} A. Constantin, J. Escher, Global existence and blow-up for a shallow water equation, Annali Sc. Norm. Sup. Pisa, vol. 26, 303--328, (1998).


\bibitem{const1998-2} A. Constantin and J. Escher, Wave breaking for nonlinear nonlocal shallow water equations, Acta Math., vol. 181, 229--243 (1998).

\bibitem{const1998-3} A. Constantin and J. Escher, Well-Posedness, Global Existence, and Blowup Phenomena, for a Periodic Quasi-Linear Hyperbolic Equation, Commun. Pure App. Math., Vol. LI, 0475--0504 (1998).

\bibitem{const2000-1} A. Constantin, Existence of permanent and breaking waves for a shallow water equation: a geometric approach, Ann. Inst. Fourier, vol. 50, 321-362, (2000).

\bibitem{const-mol}A. Constantin, L. Molinet, The initial value problem for a generalized Boussinesq equation, Differential Integral Equations, vol. 15, 1061--1072,  (2002).

\bibitem{pri-book} P. L. da Silva and I. L. Freire, On the group analysis of a modified Novikov equation, in Interdisciplinary Topics in Applied Mathematics, Modeling and Computational Science, Springer Proceedings in Mathematics $\&$ Statistics 117, (2015), DOI 10.1007/978-3-319-12307-3$\_$23.

\bibitem{pri-aims} P. L. da Silva and I. L. Freire, An equation unifying both Camassa-Holm and Novikov equations, Discreted and Continuous Dynamical Systems, 304--311, (2015), DOI: 10.3934/proc.2015.0304.

\bibitem{raspa} P. L. da Silva, Classification of bounded travelling wave solutions for the Dullin--Gottwald--Holm equation, J. Math. Anal. Appl., vol. 471, 481--488, (2019), doi: 10.1016/j.jmaa.2018.10.086.

\bibitem{silvajde2019} P. L. da Silva and I. L. Freire, Well-posedness, travelling waves and geometrical aspects of generalizations of the Camassa-Holm equation, J. Diff. Equ., vol. 267, 5318--5369, (2019).

\bibitem{silva2019} P. L. da Silva and I. L. Freire, Integrability, existence of global solutions and blow-up criteria for a generalization of the Camassa-Holm equation, arXiv:1906.00304, (2019).


\bibitem{depro}A. Degasperis and M. Procesi, Asymptotic integrability, in: Symmetry and Perturbation Theory, World Scientific, 23--37, (1999). 

\bibitem{DHH} A. Degasperis, D. D. Holm and A. N. W. Hone, A new integrable equation with peakon solutions, Theor. Math. Phys., 133, 1463--1474, (2002), DOI:10.1016/S0031-8914(53)80099-6.

\bibitem{dgh} H. Dullin, G. Gottwald, D. Holm, An integrable shallow water equation with linear and nonlinear dispersion, Phys. Rev. Lett., 87, Article 194501, (2001).



\bibitem{escher} J. Escher, Breaking water waves, In: Constantin A. (eds) Nonlinear Water Waves. Lecture Notes in Mathematics, vol 2158. Springer, Cham, (2016), DOI: 10.1007/978-3-319-31462-4$\_$2.

\bibitem{fokas} A. S. Fokas and B. Fuchssteiner, Symplectic structures, their Bäcklund transformations and hereditary symmetries, Phys. D., vol. 4, 47--66, (1981).

\bibitem{folland} G. B. Folland, Introduction to partial differential equations, 2nd edition, (1995).

\bibitem{chines-arxiv} G. Gui, Y. Liu and J. Sun, A nonlocal shallow-water model arising from the full water waves with the Coriolis effect, arXiv:1801.04665 (2018).

\bibitem{gui-jnl} G. Gui, Y. Liu and T. Luo, Model equations and traveling wave solutions for shallow-water waves with the Coriolis effect, J. Nonlin. Sci., (2018), DOI: doi.org/10.1007/s00332-018-9510-x.



\bibitem{hone} A. N. W. Hone and J. P. Wang, Integrable peakon equations with cubic nonlinearity, J. Phys. A: Math. Theor., 41, 372002, (2008).

\bibitem{hunter} J. K. Hunter, B. Nachtergaele, \textit{Applied analysis}. Singapore, World Scientific, (2005).


\bibitem{john} R. S. Johnson, Camassa–Holm, Korteweg–de Vries and related models for water waves, J. Fluid. Mech., 63--82, vol. 455, (2002).

\bibitem{kato} T. Kato, Quasi-linear equations of evolution, with applications to partial differential equations. in:
Spectral theory and diOerential equations, Proceedings of the Symposium Dundee, 1974, dedicated to
Konrad Jrgens, Lecture Notes in Math, Vol. 448, Springer, Berlin, 1975, pp. 25–70.

\bibitem{lenjde2013} J. Lenells and M. Wunsch, On the weakly dissipative Camassa--Holm,
Degasperis--Procesi, and Novikov equations, J. Diff. Equ., vol. 255, 441-448, (2013).

\bibitem{linares} F. Linares and G. Ponce, Introduction to Nonlinear Dispersive Equations, Springer, (2015).

\bibitem{liu2011} X. Liu and Z. Yin, Local well-posedness and stability of peakons for a generalized Dullin--Gottwald--Holm equation, Nonlin. Anal., vol. 74, 2497--2507, (2011).



\bibitem{mustafa} O. G. Mustafa, On the Cauchy problem for a generalized
Camassa--Holm equation, Nonlin. Anal., vol. 64, 1382--1399, (2006).

\bibitem{niujmaa2011} W. Niu and S. Zhang, Blow-up phenomena and global existence for the nonuniform weakly
dissipative $b-$equation, J. Math. Anal. Appl., vol. 374, 166--177, (2011).

\bibitem{novikov}  V. Novikov, Generalizations of the Camassa–Holm equation, J. Phys. A: Math. Theor., 42, 342002, (2009).

\bibitem{novjmp2013} E. Novruzov, Blow-up phenomena for the weakly dissipative Dullin--Gottwald--Holm equation, J. Math. Phys. 54, 092703 (2013).

\bibitem{novjde2016} E. Novruzov, Blow-up of solutions for the dissipative Dullin-Gottwald-Holm equation with arbitrary coefficients, J. Diff. Equ., vol. 261, 1115--1127, (2016).




\bibitem{chines-jde} X. Tu, Y. Liu, C. Mu, Existence and uniqueness of the global conservative weak solutions to the rotation-Camassa–Holm equation, J. Diff. Equ., (2018), DOI: 10.1016/j.jde.2018.10.012.


\bibitem{blanco} G. Rodriguez-Blanco, On the Cauchy problem for the Camassa--Holm equation, Nonlinear Anal., 46, 309--327 (2001).





\bibitem{taylor} M. E. Taylor, Partial Differential Equations I, 2nd edition, Springer, (2011).

\bibitem{tiancmp2005} L. Tian, G. Gui and Y. Liu, On the well-posedness problem and the scattering
problem for the Dullin-Gottwald-Holm equation, Commun. Math. Phys., vol. 257, 667--701, (2005).

\bibitem{yindcdc2004} Z. Yin, Well-posedness, blowup, and global existence for an integrable shallow water equation, Dis. Cont. Dyn. Sys., vol. 11, 393--411, (2004).

\bibitem{wujmp2006} S. Wu, Z. Yin, Blow-up, blow-up rate and decay of the solution of the weakly dissipative Camassa-Holm equation, J. Math. Phys., vol 47, (2006), paper 013504.

\bibitem{wujde2009} S. Wu and Z. Yin, Global existence and blow-up phenomena for the weakly dissipative Camassa--Holm equation, J. Diff. Equ., vol. 246, 4309--4321, (2009).

\bibitem{zhangjmp2015} Z. Zhang, J. Huang and M. Sun, Blow-up phenomena for the weakly dissipative Dullin-Gottwald-Holm equation revisited, J. Math. Phys. 56, 092701 (2015).

\bibitem{zhoujfa2007} Y. Zhou, Blow-up of solutions to the DGH equation, J. Func. Anal., vol. 250, 227--248, (2007).

\end{thebibliography}
\end{document}